\documentclass[12pt]{article}
\usepackage{amsmath,amssymb,amsthm,cite}

\newtheorem{theorem}{Theorem}[section]
\newtheorem{thmy}{Theorem}

\newtheorem{lemma}[theorem]{Lemma}
\newtheorem{corollary}[theorem]{Corollary}

\def\barr{\begin{array}}
\def\earr{\end{array}}

\title{On groups occurring as absolute centers\\ of finite groups}
\author{Georgiana Fasol\u a and Marius T\u arn\u auceanu}
\date{October 19, 2023}

\begin{document}

\maketitle

\begin{abstract}
Given a construction $f$ on groups, we say that a group $G$ is \textit{$f$-realisable} if there is a group $H$ such that $G\cong f(H)$, and \textit{completely $f$-realisable} if there is a group $H$ such that $G\cong f(H)$ and every subgroup of $G$ is isomorphic to $f(H_1)$ for some subgroup $H_1$ of $H$ and vice versa.

Denote by $L(G)$ the absolute center of a group $G$, that is the set of elements of $G$ fixed by all automorphisms of $G$. By using the structure of the automorphism group of a ZM-group, in this paper we prove that cyclic groups $C_N$, $N\in\mathbb{N}^*$, are completely $L$-realisable.
\end{abstract}

{\small
\noindent
{\bf MSC2020\,:} Primary 20D30; Secondary 20D45, 20D25.

\noindent
{\bf Key words\,:} inverse group theory, (completely) $f$-realisable group, automorphism group, absolute center of a group, ZM-group.}

\section{Introduction}

In group theory, there are many constructions $f$ which start from a group $H$ and produce another group $f(H)$. Examples of such group-theoretical constructions are: center, central quotient, derived quotient, Frattini subgroup, Fitting subgroup, Chermak-Delgado subgroup, automorphism group, Schur multiplier, other cohomology groups, and various constructions from permutation groups. For each of these constructions, there is an inverse problem:
\begin{equation}
\mbox{Given a group } G, \mbox{ is there a group } H \mbox{ such that } G\cong f(H)?
\end{equation}

Several new results related to this problem have been obtained in \cite{1,2,10,11} for $f(H)=D(H)$, the derived subgroup of $H$. Note that in these papers the group $H$ with the property $G\cong D(H)$ has been called an \textit{integral} of $G$ by analogy with calculus. Moreover, we recall Problem 10.19 in \cite{1} that asks to classify the groups in which all subgroups are integrable.

Other results of the same type are given by \cite{7,15,21,22,23} for $f(H)=\Phi(H)$, the Frattini subgroup of $H$. In this case, there is a precise cha\-rac\-te\-ri\-za\-tion of finite groups $G$ for which (1) has solutions, namely
\begin{equation}
G\cong\Phi(H) \mbox{ for some group } H \mbox{ if and only if } {\rm Inn}(G)\subseteq\Phi({\rm Aut}(G))\nonumber
\end{equation}(see \cite{7}).

The same problem for $f(H)={\rm Aut}(H)$, the automorphism group of $H$, has been studied in \cite{16,17}. We also recall the well-known class of \textit{capable groups}, i.e. the groups $G$ such that (1) has solutions for $f(H)={\rm Inn}(H)$, the inner automorphism group of $H$. Their study was initiated by R. Baer \cite{3} and continued in many other papers (see e.g. \cite{4,8}).

Inspired by these studies, in \cite{9} we introduced the following two notions. Given a construction $f$ on groups, we say that a group $G$ is
\begin{description}
\item[\hspace{10mm}{\rm a)}]\textit{$f$-realisable} if there is a group $H$ such that $G\cong f(H)$
\end{description}and
\begin{description}
\item[\hspace{10mm}{\rm b)}]\textit{completely $f$-realisable} if there is a group $H$ such that:
\begin{itemize}
\item[{\rm i)}] $G\cong f(H)$;
\item[{\rm ii)}] $\forall\, G_1\leq G, \exists\, H_1\leq H$ such that $G_1\cong f(H_1)$;
\item[{\rm iii)}] $\forall\, H_1\leq H, \exists\, G_1\leq G$ such that $f(H_1)\cong G_1$.
\end{itemize}
\end{description}We determined completely ${\rm Aut}$-realisable groups. We also provided some results about (completely) $f$-realisable groups for $f=Z,F,M,D,\Phi$, where $Z(H)$, $F(H)$ and $M(H)$ denote the center, the Fitting subgroup and the Chermak-Delgado subgroup of the group $H$, respectively.

The current paper deals only with finite groups. We will continue the above study for $f=L$, where $L(H)$ is the absolute center of the group $H$, i.e.
\begin{equation}
L(H)=\{h\in H\mid \alpha(h)=h, \forall\, \alpha\in{\rm Aut}(H)\}.\nonumber
\end{equation}It has been introduced by Hegarty \cite{12}, together with the autocommutator subgroup of the group $H$:
\begin{equation}
K(H)=\langle h^{-1}\alpha(h)\mid h\in H, \alpha\in{\rm Aut}(H)\rangle.\nonumber
\end{equation}\newpage \noindent Hegarty proved an analogue of Schur’s theorem for the absolute center and the autocommutator subgroup, namely that if $H$ is a group such
that $H/L(H)$ is finite, then so are $K(H)$ and ${\rm Aut}(H)$. The converse of this result is also true, as it has been shown in \cite{13}.

Clearly, the absolute center of a group is abelian, since $L(H)\subseteq Z(H)$ for all groups $H$. It was determined for several classes of $p$-groups, such as for abelian $p$-groups in \cite{6}, for minimal non-abelian $p$-groups in \cite{18} or for $p$-groups of maximal class in \cite{19}. Note that in all these cases we have $$L(H)\cong C_p^k \mbox{ for certain } k=0,1,2.$$Also, absolute centers of direct products of such groups of coprime orders will be direct products of elementary abelian $p$-groups of rank $\leq 2$, according to Lemma 2.1 of \cite{18}. This leads to the following natural question:
\begin{equation}
\mbox{Which (abelian) groups occur as absolute centers of finite groups}?
\end{equation}

In what follows, we will show that finite cyclic groups satisfy this property. More precisely, the following more powerful result holds.

\begin{theorem}
All finite cyclic groups are completely $L$-realisable.
\end{theorem}

The proof of Theorem 1.1 is based on computing the absolute center of a ZM-group, that is a finite group all of whose
Sylow subgroups are cyclic. By \cite{14}, such a group is of type
\begin{equation}
{\rm ZM}(m,n,r)=\langle a, b \mid a^m = b^n = 1,
\hspace{1mm}b^{-1} a b = a^r\rangle, \nonumber
\end{equation}
where the triple $(m,n,r)$ satisfies the conditions
\begin{equation}
{\rm gcd}(m,n)={\rm gcd}(m,r-1)=1 \mbox{ and } r^n
\equiv 1 \hspace{1mm}({\rm mod}\hspace{1mm}m). \nonumber
\end{equation}Note that $|{\rm ZM}(m,n,r)|=mn$, ${\rm ZM}(m,n,r)'=\langle a\rangle$ and $Z({\rm ZM}(m,n,r))=\langle b^d\rangle$, where $d$ is the multiplicative order of
$r$ modulo $m$, i.e.
\begin{equation}
d=o_m(r)={\rm min}\{k\in\mathbb{N}^* \mid r^k\equiv 1 \hspace{1mm}({\rm mod} \hspace{1mm}m)\}.\nonumber
\end{equation}

Finally, we mention that we were unable to give a complete answer to question (2), but, based on the above results, the following conjecture seems reasonable:

\bigskip\noindent{\bf Conjecture.} {\it All finite abelian groups are completely $L$-realisable.}
\bigskip

Most of our notation is standard and will usually not be repeated here. Elementary notions and results on groups can be found in \cite{14}.
For subgroup lattice concepts we refer the reader to \cite{20}.

\section{Automorphisms of ZM-groups}

The general form of automorphisms of a metacyclic group has been determined in the main result of \cite{5}. In our case, we get:

\begin{theorem}
Each automorphism of ${\rm ZM}(m,n,r)$ is given by
\begin{equation}
b^ua^v\mapsto b^{yu}a^{x_1v+x_2[u]_r},\, u,v\geq 0,\nonumber
\end{equation}for a unique triple of integers $(x_1,x_2,y)$ such that
\begin{equation}
0\leq x_1,x_2<m,\, (x_1,m)=1,\, 0\leq y<n \mbox{ and } y\equiv 1 \hspace{1mm}({\rm mod}\hspace{1mm}d),\nonumber
\end{equation}where
\begin{equation}
[u]_r=\left\{\barr{lll}
    \!\!1+r+\cdots+r^{u-1},&u>0\\
    \!\!0,&u=0.\earr\right.\nonumber
\end{equation}In particular, we have
\begin{equation}
|{\rm Aut}({\rm ZM}(m,n,r))|=m\varphi(m)\frac{n}{d}\,.\nonumber
\end{equation}
\end{theorem}

Since ${\rm Inn}({\rm ZM}(m,n,r))$ is of order $md$, Theorem 2.1 leads to the following corollary.

\begin{corollary}
We have
\begin{equation}
|{\rm Out}({\rm ZM}(m,n,r))|=\frac{\varphi(m)n}{d^2}\nonumber
\end{equation}and
\begin{equation}
{\rm ZM}(m,n,r) \mbox{ is a complete group } \Leftrightarrow\, \varphi(m)=n=d.\,\footnote{\,We note that in this case $m$ is a cyclic number.}\nonumber
\end{equation}
\end{corollary}

The form of central automorphisms and of IA-automorphisms of ${\rm ZM}(m,n,r)$ can be also inferred from Theorem 2.1.

\begin{corollary}
Each central automorphism of ${\rm ZM}(m,n,r)$ is given by
\begin{equation}
b^ua^v\mapsto b^{yu}a^v,\, u,v\geq 0,\nonumber
\end{equation}for a unique integer $y$ such that
\begin{equation}
0\leq y<n \mbox{ and } y\equiv 1 \hspace{1mm}({\rm mod}\hspace{1mm}d)\nonumber
\end{equation}and
each IA-automorphism of ${\rm ZM}(m,n,r)$ is given by
\begin{equation}
b^ua^v\mapsto b^ua^{x_1v+x_2[u]_r},\, u,v\geq 0,\nonumber
\end{equation}for a unique pair of integers $(x_1,x_2)$ such that
\begin{equation}
0\leq x_1,x_2<m \mbox{ and } (x_1,m)=1.\nonumber
\end{equation}In particular, we have
\begin{equation}
|{\rm Aut_C}({\rm ZM}(m,n,r))|=\frac{n}{d} \mbox{ and } |{\rm Aut_{IA}}({\rm ZM}(m,n,r))|=m\varphi(m).\nonumber
\end{equation}
\end{corollary}

Next, we determine the absolute center of ${\rm ZM}(m,n,r)$.

\begin{theorem}
We have
\begin{equation}
L({\rm ZM}(m,n,r))=\langle b^{de}\rangle,\nonumber
\end{equation}where
\begin{equation}
e=\min\{s\in\mathbb{N}^*\mid d^2s\equiv 0 \hspace{1mm}({\rm mod}\hspace{1mm}n)\}.\nonumber
\end{equation}Moreover
\begin{equation}
|L({\rm ZM}(m,n,r))|=\frac{n}{(de,n)}=\frac{\frac{n}{d}}{(e,\frac{n}{d})}\,.\nonumber
\end{equation}
\end{theorem}

\begin{proof}
Since $L({\rm ZM}(m,n,r))\subseteq Z({\rm ZM}(m,n,r))$, it follows that
\begin{equation}
L({\rm ZM}(m,n,r))=\langle b^{de}\rangle \mbox{ for some } e\in\mathbb{N}^*.\nonumber
\end{equation}By Theorem 2.1, the condition $b^{de}\in L({\rm ZM}(m,n,r))$ means
\begin{equation}
b^{yde}a^{x_2[de]_r}=b^{de},\nonumber
\end{equation}that is
\begin{equation}
n\,|\,de(y-1) \mbox{ and } m\,|\,x_2[de]_r
\end{equation}for all pairs of integers $(x_2,y)$ satisfying $0\leq x_2<m$, $0\leq y<n$ and $y\equiv 1 \hspace{1mm}({\rm mod}\hspace{1mm}d)$. Clearly, the relations (3) hold if and only if they hold for $x_2=1$ and $y=d$, that is
\begin{equation}
n\,|\,d^2e \mbox{ and } m\,|\,[de]_r.\nonumber
\end{equation}Since $m\,|\,[de]_r$ is true for all $e\in\mathbb{N}^*$, we must have only $n\,|\,d^2e$ and the proof is completed by the remark that $e$ must be chosen minimal with respect to this condition.
\end{proof}

For example, for ${\rm ZM}(5,16,2)=SmallGroup(80,3)$ we get $d=4$, $e=1$ and so $L({\rm ZM}(5,16,2))=Z({\rm ZM}(5,16,2))\cong C_4$, while for ${\rm ZM}(5,48,2)=SmallGroup(240,5)$ we get $d=4$, $e=3$ and so $L({\rm ZM}(5,48,2))\cong C_4$. Note that in the second case we have $L({\rm ZM}(5,48,2))\neq Z({\rm ZM}(5,48,2))\cong C_{12}$.\footnote{More precisely, since ${\rm ZM}(5,48,2)\cong C_3\times{\rm ZM}(5,16,2)$, we have $L({\rm ZM}(5,48,2))\cong L(C_3)\times L({\rm ZM}(5,16,2))\cong 1\times C_4\cong C_4$ and $Z({\rm ZM}(5,48,2))\cong Z(C_3)\times Z({\rm ZM}(5,16,2))\cong C_3\times C_4\cong C_{12}$.}

\section{Proof of the main result}

First of all, we will show that $L({\rm ZM}(m,n,r))$ can be any cyclic group $C_N$ when $N=q^{\alpha}$ is a prime power. By Theorem 2.4, we must find a triple of positive integers $(m,n,r)$ such that
\begin{equation}
\left\{\barr{lll}
    \!\!{\rm gcd}(m,n)={\rm gcd}(m,r-1)=1\\
    \!\!r^n\equiv 1 \hspace{1mm}({\rm mod}\hspace{1mm}m)\\
    \!\!\frac{n}{(de,n)}=N,\earr\right.
\end{equation}where
\begin{equation}
\left\{\barr{lll}
    \!\!d=o_m(r)\\
    \!\!e=\min\{s\in\mathbb{N}^*\mid d^2s\equiv 0 \hspace{1mm}({\rm mod}\hspace{1mm}n)\}.\earr\right.\nonumber
\end{equation}
\smallskip

We start with the following auxiliary result.

\begin{lemma}
There exist a prime $p$ and a positive integer $r$ such that $o_p(r)=q^{\alpha}$.
\end{lemma}

\begin{proof}
By Dirichlet's Theorem, the arithmetic progression $1+tq^{\alpha}$, $t\in\mathbb{N}$, contains an infinite number of primes. Let $p$ be one of them. Then $q^{\alpha}$ divides $p-1$. Since the multiplicative group of integers modulo $p$ is cyclic of order $p-1$, it contains elements of order $q^{\alpha}$. Thus there is $r\in\mathbb{N}^*$ such that $o_p(r)=q^{\alpha}$, as desired.
\end{proof}

Clearly, the triple $(p,q^{2\alpha},r)$, where $p$ and $r$ are given by Lemma 3.1, satisfies the conditions (4), that is we have
\begin{equation}
L({\rm ZM}(p,q^{2\alpha},r))\cong C_{q^{\alpha}}.\nonumber
\end{equation}

For the proof of Theorem 1.1, we also need the following corollary.

\begin{corollary}
Let $P$ be the set of all primes. Then for any $q_1, ..., q_k\in P$ and any $\alpha_1, ..., \alpha_k\in\mathbb{N}^*$, there exist distinct primes $p_1, ..., p_k\in P\setminus\{q_1,...,q_k\}$ and positive integers $r_1, ..., r_k$ such that $o_{p_i}(r_i)=q_i^{\alpha_i}$, $i=1,...,k$.
\end{corollary}

\begin{proof}
We apply Lemma 3.1 for all $i=1,...,k$. It suffices to observe that the primes $p_1$, ..., $p_k$ can be chosen to be distinct, and different from $q_1, ..., q_k$, by Dirichlet's Theorem.
\end{proof}

We are now able to prove our main result.

\begin{proof}[Proof of Theorem 1.1.]
Let $N\geq 2$ be an integer and $N=q_1^{\alpha_1}\cdots q_k^{\alpha_k}$ be the decomposition of $N$ as a product of prime factors. We choose $p_1$, ..., $p_k$ and $r_1$, ..., $r_k$ as in Corollary 3.2, and let
\begin{equation}
H_i={\rm ZM}(p_i,q_i^{2\alpha_i},r_i),\, i=1,...,k.\nonumber
\end{equation}Then, for each $i$, we have
\begin{equation}
L(H_i)\cong C_{q_i^{\alpha_i}}.\nonumber
\end{equation}

Let $H=H_1\times\cdots\times H_k$. Since the groups $H_i$, $i=1,...,k$, are of coprime orders, we infer that they are characteristic in $H$ and therefore
\begin{equation}
L(H)\cong L(H_1)\times\cdots\times L(H_k)\cong C_{q_1^{\alpha_1}}\times\cdots\times C_{q_k^{\alpha_k}}\cong C_N,\nonumber
\end{equation}i.e. $C_N$ is $L$-realisable.

Obviously, every subgroup $G_1$ of $C_N$ is of type $C_{N_1}$, where $N_1=q_1^{\beta_1}\cdots q_k^{\beta_k}$ and $0\leq\beta_i\leq\alpha_i$, $i=1,...,k$. It is easy to see that each group $H_i$ contains a normal subgroup $S_i\cong{\rm ZM}(p_i,q_i^{\alpha_i+\beta_i},r_i)$. Since all normal subgroups of a ZM-group are characteristic, it follows that $S_i$ is characteristic in $H_i$, and so in $H$. Moreover, we have $L(S_i)\cong C_{q_i^{\beta_i}}$. Then the subgroup $S=S_1\times\cdots\times S_k$ of $H$ satisfies $L(S)\cong G_1$.

Conversely, we observe that every subgroup $T$ of $H$ is a direct product $T_1\times\cdots\times T_k$, where $T_i\leq H_i$, $i=1,...,k$. Then each $T_i$ is either cyclic or of type ${\rm ZM}(p_i,q_i^{\gamma_i},r_i)$ with $\alpha_i\leq\gamma_i\leq 2\alpha_i$. In the second case, we get
\begin{equation}
L(T_i)\leq Z(T_i)\cong C_{q_i^{\gamma_i-\alpha_i}}.\nonumber
\end{equation}Consequently, $L(T_i)$ is either trivial or isomorphic to a subgroup of $C_{q_i^{\alpha_i}}$. It follows that $L(T)\cong L(T_1)\times\cdots\times L(T_k)$ is isomorphic to subgroup of $C_{q_1^{\alpha_1}}\times\cdots\times C_{q_k^{\alpha_k}}\cong C_N$.

This completes the proof.
\end{proof}

\vspace*{3ex}
\small

\begin{minipage}[t]{7cm}
Georgiana Fasol\u a \\
Faculty of  Mathematics \\
"Al.I. Cuza" University \\
Ia\c si, Romania \\
e-mail: \!{\tt georgiana.fasola@student.uaic.ro}
\end{minipage}
\hfill\hspace{20mm}
\begin{minipage}[t]{5cm}
Marius T\u arn\u auceanu \\
Faculty of  Mathematics \\
"Al.I. Cuza" University \\
Ia\c si, Romania \\
e-mail: \!{\tt tarnauc@uaic.ro}
\end{minipage}

\end{document}